\documentclass[12pt]{article}
\newcommand{\average}{-\!\!\!\!\!\!\int}

\newcommand{\comment}[1]{}

\newcommand{\diam}{\mathop{\rm diam}\nolimits}
\newcommand{\dist}{\mathop{\rm dist}\nolimits}
\renewcommand{\div}{{\mathop{\rm\,div\,}\nolimits}}

\providecommand{\qed}{\vrule height 6pt depth 0pt width 3 pt}

\newcommand{\reals}{{\bf R}}

\newcommand{\BibTeX}{{\rm B\kern-.05em{\sc i\kern-.025em b}\kern-.08em     
    T\kern-.1667em\lower.7ex\hbox{E}\kern-.125emX}}

\usepackage{hyperref}

\providecommand\note{}
\renewcommand\note[1]{}
\renewcommand\marginpar[1]{}

\usepackage{amsmath,amsthm,amscd,amssymb}
\usepackage{xcolor}

\usepackage{chngcntr}
\usepackage{apptools}
\AtAppendix{\counterwithin{equation}{section}}

\newcommand{\sobolev}[2]{  W ^ {{#2},{#1}}}
\newcommand{\ball}[2]{B_{#2}(#1)}

\newcommand{\locdom}[2]{{\Omega_{#2}(#1)}}
\newcommand{\sball}[2]{{\Delta_{#2}(#1)}}

\newenvironment{example}[1][Example]{\begin{trivlist}\item[\hskip \labelsep
{\it #1. }]}{  \goodbreak \end{trivlist}}   
\numberwithin{equation}{section} 
\newtheorem{theorem}[equation]{Theorem}

\newtheorem{lemma}[equation]{Lemma}

\begin{document}
\title{Estimates for the $L^q$-mixed problem in $C^{1,1}$-domains}

\author{
 R.M.~Brown\footnote{
Russell Brown   is  partially supported by  grants from the Simons
Foundation (\#195075,\#422756).
} \\ Department of Mathematics\\ University of Kentucky \\
Lexington, KY 40506-0027, USA
\and
L.D. Croyle
\\ Department of Mathematics\\ Baldwin-Wallace University\\
Berea, OH 44017
}

\date{}

\maketitle

\begin{abstract}
We consider the $L^q$-mixed problem in domains in $\reals ^n$ with $C^ {1,1}$-boundary. We
assume that the boundary  between the sets where we specify Neumann and
Dirichlet data is Lipschitz. With these assumptions, we show that we
may solve the $L^q$-mixed problem for $q$ in the range $1<q <
n/(n-1)$. 
\end{abstract}

\section{Introduction}\label{Introduction}
The goal of this note is to establish a regularity result for the
$L^q$-mixed problem. Our work builds on an earlier result of Ott and Brown
\cite{MR3042705}
which establishes existence and uniqueness for the
$L^q$-mixed problem for $ q $ near 1. 
In this paper, we consider a
more restrictive class of domains than was considered in Ott and
Brown, but we are able to give an explicit range of exponents $q$ for
which we can solve the mixed problem. This range is easily seen to be
sharp in two dimensions.
The new ingredient in this work compared to Ott and Brown's work is  a
result of
Savar\'e \cite{GS:1997}. Savar\'e's result is a regularity result for
solutions of 
the mixed problem in a Besov space. We use his result to prove a reverse H\"older
inequality. This inequality then feeds into the machinery of  Ott and
Brown to obtain our main theorem.

We let $ \Omega \subset \reals ^n$
be a bounded open set and suppose the boundary $ \partial \Omega$ is
partitioned into two sets $D$ and $N$. We assume that we are given
functions $ f_D$ and $f_N$ defined on $D$ and $N$, respectively. 
By the $L^q$-mixed problem, we mean the problem of finding a function
$u$ which
satisfies
\begin{equation}
  \label{MP}
  \left \{   \begin{aligned}
&    \Delta u = 0, \qquad &&\mbox{in } \Omega \\
&    u = f_D , \qquad && \mbox {on } D \\
 &   \frac { \partial u } { \partial \nu } = f_N, \qquad &&\mbox {on }
    N \\
&\nabla u ^* \in L^ q( \partial \Omega)   \qquad & \   
  \end{aligned}
  \right.
\end{equation}
Our assumptions on the domain are below. In particular, our hypotheses
will imply that 
 the surface measure on $ \partial \Omega$ is
defined. We use $ \nabla u^*$ to denote the non-tangential maximal
function and this will be defined in section \ref{Prelim}  below.

Our main result is  the following theorem. See section \ref{Prelim}
for definitions of several of the objects appearing in the theorem. 

\begin{theorem}\label{Main} Suppose that $ \Omega\subset \reals ^n$,
  $N$, and $D$ is a  standard
  $C^{1,1}$-domain for the mixed problem  as defined in section
  \ref{Prelim}. Suppose that $ q \in (1,
  n/(n-1))$,  that $f_N$ is  in $L^q(N)$ and $f_D$ is in the
  Sobolev space $\sobolev  q1(D)$. Under these assumptions there exists a
  unique solution of the $L^q$-mixed problem for the Laplacian \eqref{MP}
  and the solution  satisfies the estimate
  $$
  \| \nabla u ^* \|_{ L^q ( \partial \Omega)}  \leq C   
[    \| f_N\|_{ L^q(N)} + \| f_D \| _{\sobolev q1 (D)}].
  $$
\end{theorem}

%%Need to define simple near here. ???
We next recall a well-known example that shows that, at least in two
  dimensions, the range of exponents in Theorem \ref{Main} is sharp.
\begin{example}  
  We
let $ \Omega \subset \{ (x_1, x_2) : x _2 >0 \}$ be a  smooth domain
with $ [-1,1] \times \{ 0 \} \subset \partial \Omega$. We define $N$
and $D$ by $D= [0,1] \times \{0\}$ and $N= \partial \Omega \setminus
D$. 
Consider the function $u$ defined in polar  coordinates by
\begin{equation} \label{Simple} u(r,
\theta) = r ^ { 1/2 } \cos( \theta /2). 
\end{equation}
The function $u$ will solve the mixed problem in $ \Omega$ with $ f_N$
bounded and $f_D =0$. However,  we have $ |\nabla u (x) | = c|x|^ {
  -1/2} $ and thus we have $ \nabla u ^ * \in L^ q( \partial \Omega)$
precisely if $ q < 2$.\hfill \qed
\end{example}

\section{Definitions and preliminary results} \label{Prelim}

\note{{\color{blue}Changed to require  dissection condition for $x$ in $ \Lambda$.}}
In this section we give the main definitions used in the statement and
proof of our main result.
We begin by defining the domains we will use.  We assume that $
\Omega \subset \reals ^n$ is a bounded, connected,  and open set and that
the boundary is $C^{1,1}$. This will mean that there exists $ r _0$
and $M$ so that for each $ x \in \partial \Omega$ we may find
coordinates $ ( y', y _n )=(y_1, y'',y_n)  \in \reals \times \reals ^ { n-2} \times \reals$  and a
$C^ { 1,1}$-function $\phi : \reals ^ { n-1} \rightarrow  \reals $ so that we have
\begin{equation} \label{Coords}
  \begin{aligned}
    \Omega \cap B_ { 100 r_0 } (x)
   & = \{ (y', y _n ) : y _ n > \phi ( y') \} \cap B _ { 100 r_0}
    (x) \\
\partial \Omega \cap B_ { 100 r_0 } (x)
&= \{ (y', y _n ) : y _ n = \phi ( y') \} \cap B _ { 100 r_0} (x).
  \end{aligned}
\end{equation}
Here, we are using $ B_r(x)$ to denote a ball with radius $r$ and
center $x$. 
  To prove our regularity result, we will need to impose conditions on
  the boundary between $D$ and $N$. We let $ \Lambda $ denote the
  boundary of $D$ relative to $ \partial \Omega$ and for each $ x \in
  \Lambda$, 
  we assume that with the coordinate system and $ \phi$ as above, we
  also have 
  a Lipschitz function  $ \psi : \reals ^ {n-2} \rightarrow \reals $ 
so that 
  \begin{equation}\label{Bcoords}
    \begin{aligned}
      D \cap B_ {100 r_0 }( x) =
      \{ (y _1, y '', y_n ) \in \reals \times 
      \reals ^ {n-2} \times \reals : y_ n = \phi ( y'), y_1  \geq
      \psi (y'') \} \cap B _ {100r_0}(x) \\
      N \cap B_ {100 r_0 }( x) =
      \{ ( y _1, y '', y_n ) \in \reals \times 
      \reals ^ {n-2} \times \reals : y_ n = \phi (y'), y_1  <
      \psi ( y'') \}\cap B _{100r_0}(x).
    \end{aligned}
  \end{equation}
In both \eqref{Coords} and \eqref{Bcoords}, we require that the
coordinate system be a rigid motion of the standard coordinates on $
\reals ^n$ and that the functions $ \phi$ and $\psi$ satisfy the
conditions
\begin{equation} \label{FunBounds}
 \| \nabla \phi \| _ { L^ \infty ( \reals ^ { n-1})} + 
r_0 \| \nabla ^2 \phi \|_{L^ \infty ( \reals ^ { n-1})}
\leq M, \qquad  \| \nabla  \psi \|_{L^ \infty ( \reals ^ {n-2})} \leq
M.
\end{equation}
We will call $ \Omega $, $N$, and $D$ a {\em standard $C^{1,1}$-domain for
the mixed problem}.  We will use $r_0$ as a characteristic length for
the domain.   Our goal is provide
results which are scale-invariant and this is the reason for the
appearance of $r_0$ in (\ref{FunBounds}). 

\note{{\color{blue} Changed Sobolev spaces to $W$ to be consistent. }}

Next, we define Sobolev spaces on $ \Omega$. For $ 1\leq p <
\infty$, we let $\sobolev p 1( \Omega)$ be the standard Sobolev space of
functions
with one derivative in $L^p( \Omega)$. For $ D \subset \partial \Omega$
and $p < \infty$, we define $W^ {1,p}_D( \Omega)$ to be the closure of
$ C^ \infty _D ( \bar  \Omega)$ in $ W^ { 1,p }( \Omega)$. Here, $ C^
\infty _D( \bar \Omega)$ denotes the collection of functions in $ C^
\infty ( \bar \Omega)$ 
which vanish on a neighborhood of $ \bar D$, the closure of $D$.
In order to make our estimates scale-invariant, we will define the
norm on  $W^ {
  1,p} _D(\Omega) $ as
$$
\|u \| _ { W^ { 1,p} ( \Omega) } = \left( \int _ \Omega |\nabla u | ^ p + r_0
^ { -p } |u|^p \, dx\right) ^ { 1/p}.
$$
We will use $ W^ {- 1,2} _D( \Omega)$ to denote the dual  of the space
$W^ { 1,2}_D( \Omega)$. 
\note{ Clarify why we need $ \bar D$. 

}

We will use  $\sobolev q1(D)$ to denote  the Sobolev space of functions in
$L^q(D)$  which also have one derivative in $L^q(D)$.  This space will
be defined as the restriction of $\sobolev q 1 (\partial \Omega)$ to
the closed set $ D$. 
In order to make our estimates scale
correctly, we will use the norm
$$
\|u \| _{ L^q_1(D)} = \left( \int_D |\nabla_{tan}u |^q + r_0^ { -q} |u|^q \,
d\sigma\right ) ^ { 1/q}.
$$
See \cite[p.~1337]{MR3042705} or \cite[p.~580]{GV:1984} for a
definition of this space and the tangential gradient, $ \nabla
_{tan}$. 

We let $ W^ { 1/2, 2} ( \partial \Omega)$ denote the image of 
$W^{1,2} ( \Omega)$ under the trace map. Similarly, $W^ {1/2, 2}_D(
\partial \Omega)$ will denote  the  image of $ W ^ { 1,2 } _D (
\Omega)$ under the trace map. We denote the dual of $ W^ { 1/2,2} _D(
\partial \Omega)$  by $ W^ { -1/2,2 } _D( \partial \Omega)$.
The space $ W^ { -1/2, 2 }_D( \partial \Omega)$ is a natural space for
Neumann data for the weak  mixed problem. The Dirichlet data in the
mixed problem will be the  restriction to $D$ of an  element in  $W^ { 1/2, 2}( \partial
\Omega)$.

\note{ Need to specify norms on trace spaces/note hilbert structure.

{\color{blue} Corrected space for Dirichlet data. }}

While our main result is for the Laplacian, at one point in the
argument of Savar\'e it is convenient to flatten the boundary using a
$ C^ { 1,1 }$-diffeomorphism. Pulling back a harmonic function in $ \Omega$ to
a domain with a flat boundary will produce a function which solves an
elliptic equation with Lipschitz coefficients.

\note{ {\color{blue}Domain of $x$ and $y$ may vary, hard to specify.

  We can state problem for a general domain, rather than specializing
  to $ \reals ^n_+$.}}
We let $L$ denote an operator $L= \div A \nabla$ where the
coefficient matrix $A$ is symmetric,  Lipschitz, and elliptic.  We will have the
quantitative assumptions
\begin{align}
  &  |A(x) -A(y) | \leq M |x-y|, \label{Lipschitz} \\
  & M^{-1} |\xi|^2 \leq A(x) \xi \cdot \xi \leq M|\xi|^2 ,
  \qquad \xi \in \reals ^ n .  \label{Elliptic} 
\end{align}
We will consider the problem
\begin{equation} \label{MP-A}
\left \{
\begin{aligned} & \div A \nabla u = F, \qquad & \mbox{in } \Omega \\
  & A \nabla u \cdot \nu = f_N, \qquad& \mbox{on } N \\
  & u = f _D, \qquad & \mbox{on } D. 
\end{aligned}
\right.
\end{equation}
We will generally work with the weak formulation of the problem
\eqref{MP-A} which is 
\begin{equation} \label{WMP}
  \left\{   \begin{aligned}
&\int_\Omega A\nabla u \cdot \nabla \phi \, dx + \langle F, \phi
\rangle _\Omega = \langle f_N, \phi \rangle _{\partial \Omega}  ,
\qquad \phi \in W^ {1,2} _D( \Omega) \\
&u-f_D \in W^ {1,2} _D ( \Omega).
  \end{aligned}
  \right.
\end{equation}
\note{ Need to make sure  that we have the right space for $F$. }
In the statement (\ref{WMP}), we are using $f_D$  to denote a function
in $W^{1,2}(\Omega)$ as well as the boundary values in
$W^{1/2,2}(\partial \Omega)$. 
The forcing term $F$ will lie in $W^ { -1,2}_D( \Omega)$
and the Neumann data will come from $ W ^ { -1/2,2}_D(
\partial \Omega)$.  We use $ \langle \cdot , \cdot \rangle _ \Omega$
to denote the pairing between $ W^ { -1,2}_D( \Omega)\times W ^ {1,2
}_D( \Omega) \rightarrow \reals$ and $\langle \cdot, \cdot \rangle_{ \partial  \Omega}$
for the pairing $ W ^ { -1/2,2}_D( \partial \Omega) \times W^ { 1/2,
  2}_D( \partial \Omega) \rightarrow \reals$.  We observe that we have the Sobolev embedding of $W^ {
  1/2, 2}_D(\partial \Omega)$ into $L^ { 2( n-1)/(n-2)}( \partial
\Omega)$ when $n\geq 3$. Thus,  passing to the dual implies $ L^ p (N)
\subset W^ { -1/2, 2 }_D( \partial \Omega)$ if $p \geq 2 ( n-1) /n$.
When $n=2$, we have that $ 2(n-1)/(n-2)= \infty$ and the embedding of $W^ { 1/2, 2}_D( \partial \Omega)$ into
$L^\infty ( \partial \Omega)$ fails. However, we do  have the
embedding  $W^
{1/2,2}_D(\partial \Omega) \subset L^p( \partial \Omega)$  for $
p < \infty$ and thus 
  we still obtain the
embedding  $L^ p (N) \subset W^ { -1/2,2}_D(\partial \Omega)$ for
$p>1$.

To estimate solutions of the mixed problem when $f_N$  comes from
$L^p(N)$, we will use the non-tangential maximal function. For a
function $u$ on $ \Omega$ taking values in  $ \reals ^d$ for some $d$, we define the non-tangential maximal
function $u ^ *$ by
$$
u^* (x) = \sup_{ y \in \Gamma (x)} |u(y)|.
$$
In this definition, $ \Gamma (x)$ is a non-tangential approach region
defined by
$$
\Gamma (x ) = \{ y : |x-y | < (1+\alpha ) \dist(y, \partial \Omega)\}
, \qquad x \in \partial \Omega 
$$
where $ \alpha > 0$ is fixed.
While $ u^*$ depends on $ \alpha$, the $L^p$-norms of non-tangential
maximal functions defined using different values of $\alpha$ will be
comparable. Thus we suppress the value of $ \alpha $ in our notation.

Our main argument will consider a number of local estimates. For these
estimates, we will use surface balls $\sball x r  = B_r (x) \cap
\partial \Omega$.  We will also need local domains $\locdom x r  =
\Omega \cap B_r (x)$. Both objects will be defined for $ x \in
\partial \Omega$ and $ 0 < r < r_0$. 
In our estimates, we will allow the constants to depend on $M$, the
constant which appears in our definition of the domain, 
and the $L^q$-exponents. 

\note{ Have we defined the meaning of $ \approx$.
{\color{blue}
To fix Ott and Brown, we need to worry about Lipschitz domains, but I
should not bring it up here. }

Describe dependence of constants.

Fix a collection of coordinate balls. 
}

\section{A reverse H\"older inequality at the boundary}

The new ingredient in this work as compared to the earlier work of Ott
and Brown \cite{MR3042705} is a reverse H\"older inequality with a larger
range of exponents.  This inequality follows from a regularity  estimate of
Savar\'e \cite{GS:1997}  for the mixed problem. The example in 
\eqref{Simple} %%See page 10
shows that Savar\'e's regularity result is sharp in the scale of Sobolev
spaces when $n=2$. In addition, it shows that the upper bound on the
exponent for 
the mixed problem is sharp as well. 
%%page 32.5 of ms.
\note{ Can we show optimality of besov space result in higher
  dimensions?

{\color{orange}  Modified definition of weak solution. It seemed
  easier to extend the test function than to restrict $f_N$. }

}

To prove a local estimate, it will be helpful to have a definition of
a weak solution with boundary data specified on only part of the
boundary. Thus, if $ \Omega$ is a domain and $D$, $N \subset \partial
\Omega$ are a decomposition of the boundary, we say that $u$ solves
the local mixed problem
\begin{equation}\label{LWS} 
  \left\{
  \begin{aligned} &\div A   \nabla u =0 , \qquad &&\mbox{in }
 \locdom x r   \\
 &u = 0 , \qquad &&\mbox{on } D  \cap \locdom x r\\
 &\frac  {\partial u  }{ \partial \nu } = f_N, \qquad &&\mbox{on }
  N \cap \locdom x r 
\end{aligned}
\right.
\end{equation}
if
$$
\int _ {\Omega}  A \nabla u \cdot \nabla \phi \, dy
= \langle f_N , \phi \rangle _ { \partial \Omega} , \qquad \phi \in W^
{ 1,2 } _ {  D'} ( \locdom x r )
$$
where $  D' = (D \cap  \bar B _r (x)  \cup  ( \partial \ball x
r \cap \bar \Omega ) . $  Note that if $ \phi \in W_{D'}^ { 1,2}( \locdom x
r)$, then
%%Bergh and Lofstrom or ????
Following Stein \cite[p.~152]{ES:1970}, we introduce the Besov spaces on $ \reals ^ n $, $B^
s_{ p, q}$, $ 1 \leq p,q\leq \infty$ and $ 0 <s<1$. We let
$$
\Delta _h (x) = u(x+h) -u(x), \qquad \mbox{for } x,h \in \reals^n.
$$
The norm for $B ^ s _ { p,q} $ is defined by
$$
\| u \| _ { B_{ p,q}^s} = \| u \| _ { L^ p } + \left (\int _{ \reals ^
  n } \left ( \frac { \| \Delta _h u \| _ { L^ p }} { |h|^ s }
\right)^ q \frac { dh}{ |h|^ n }\right ) ^ { 1/q}
$$
for $ q < \infty $. When $ q = \infty$, we set
$$
\| u \| _ {B^s _ { p, \infty}} = \| u \| _ { L^ p } + \sup _{ h \in
  \reals ^ n } \frac { \|\Delta _h u \| _ { L^ p } }{ |h|^ s}
$$
For a domain $ \Omega$, we let $ B^ s _ { p, q } (\Omega )$  be the
image of $ B^ s _ { p, q} ( \reals ^ n )$ under the restriction map, $
u \rightarrow u |_ \Omega$.

\note{ Need to note that $W^ { 1/2, 2}_D$ is  $B^ { 1/2} _{2,2}$.

  {\color{orange} Yes, you mentioned the problem of two uses of
    $\Delta$. I am going to take the position that this won't cause
    confusion. Sorry I neglected to respond to this earlier. }
}

We will localize a solution to a neighborhood of a point on the
boundary and apply a change of variables to obtain a problem in a
half-space. It will be an important point that we have uniform
estimates for the family of problems that arise from this
procedure. We will use $M$ in the quantitative estimates for the
inputs to these problems and obtain estimates which depend on the
problem through $M$.

We let $ A(x)$ be a symmetric matrix which satisfies the Lipschitz and
ellipticity
 conditions,   \eqref{Lipschitz} and \eqref{Elliptic}.
We let $ \psi : \reals ^ { n-2} \rightarrow  \reals  $ be a Lipschitz
function with constant $M$ and assume that
$$
\begin{aligned}
    D & = \{ ( x_1, x'', 0 ) : x_1 \geq \psi ( x'')\} \\
    N & = \{ ( x_1, x'', 0 ) : x_1 < \psi ( x'')\}.  
\end{aligned}
$$
With $A$, $N$, and $D$ as above, we consider the mixed problem on $
\reals ^ n _+$,
\begin{equation}
\label{HWMP}
\left \{ \begin{aligned}
 & \div A \nabla u -u = F, \qquad &&\mbox{in } \reals ^ n _+\\
  &  u = 0 , \qquad &&\mbox{on } D \\
&  A \nabla u \cdot \nu = f_N, \qquad &&\mbox{on } N
\end{aligned}
\right.
\end{equation}
and recall a regularity result for this problem.
\begin{theorem}[Savar\'e] \label{Savare} Let $u$ solve (\ref{HWMP}), then we
  have a constant $C$ so that 
  $$
  \|\nabla u \| _{B^ { 1/2} _{2, \infty } ( \reals ^ n _+) }
    \leq C [ \| F \|_{ L^2 ( \reals ^ n _+)}
        + \| f _N \| _{ W^ { 1/2, 2}_D( \reals ^ { n-1}) }].
        $$
        The constant in this estimate depends only on $M$ and the
        dimension $n$. 
\end{theorem}

\note{{\color{orange} Sentence rewritten.}}
Note that in this theorem we are assuming that $ f_N \in W^ { 1/2,
  2}_D( \reals ^ {n-1})$. The  function $f_N$ defines an element of $ W^ { -1/2,
  2} _D( \reals ^ { n-1})$ by
$$
\phi \rightarrow \int _N  f_n \phi \, d \sigma.
$$
%% p. 40.
Theorem \ref{Savare} is a small extension of the result stated by Savar\'e
\cite{GS:1997}. The difference is that we allow a more general
separation between $D$ and $N$. Our condition that $ \psi$ is
Lipschitz implies that we have that $ h+ D \subset D$ for $h$ in
an open cone in $ \reals ^ {n-1}\times \{0\}$. Since a cone in $
\reals ^ { n-1} \times \{ 0 \}$, contains a basis of $ \reals ^ { n-1}
\times \{0 \}$ we are able to carry through the argument on page 882
of Savar\'e's work \cite{GS:1997}. The dissertation of Croyle 
 \cite[p.~16]{MR3564069} provides more details.

We are now ready to state the reverse H\"older inequality.
\begin{theorem} \label{RHI} Suppose $ \Omega$, $N$, and $D$ is a standard $C^ {
    1,1}$-domain for the mixed problem. Fix $0< r < r_0$  and let $u$ be
  a solution to
  $$
  \left \{
  \begin{aligned}&\Delta u =0 , \qquad & & \mbox{in } \Omega _ { 2r } (x)
    \\
 &   u = 0 , \qquad && \mbox{on } D \cap \bar \Omega _ { 2r } (x) \\
&    \frac { \partial u } { \partial \nu } = 0 , \qquad &&  \mbox{on } N
    \cap \bar \Omega _ { 2r } (x)
  \end{aligned}
  \right .
  $$
  as in \eqref{LWS}. 

  Then for $1< p < 2n / ( n-1) $, we have
  $$
  \left ( \average _ { \Omega _ { r } (x)} |\nabla u |^ p \,dy \right)
    ^ { 1/ p } \leq C \average _ { \Omega _ { 2r} (x) }
    |\nabla u | \, dy . 
    $$
The constant in this estimate depends on $M$ and $p$.     
\end{theorem}
%%40

\begin{proof} We may rescale and translate the coordinates so that $
  r=1$ and $x = 0$. We choose a cutoff function $ \eta$ which is
  supported in $ B_2(0)$ and is equal to one on $B_1(0)$. We assume
  that $ \partial \Omega = \{ y : y _n = \phi(y')\}$ in a neighborhood
  of $ \bar B _2 (0)$ and let $ \Phi (y) = ( y', \phi (y') + y _ n )$
  map a neighborhood of 0 in $ \reals ^ n _+$ onto $ B _ 2 (0) \cap
  \Omega = \Omega _ 2 (0)$. Furthermore, we set $ \tilde N = ( \reals ^
         { n-1} \times \{ 0 \}) \cap \{ y : y _1 < \psi (
         y'') \}$ and $ \tilde D =  (\reals ^ { n-1 } \times \{ 0 \})
         \cap \{ y : y_1 \geq \psi ( y'') \}$.
         We define $v$ by 
$     v(y) = (\eta  ( u- \bar u) ) \circ \Phi (y)$
         where
         $$\bar u      = \left \{ \begin{aligned}& 0 , \qquad  && \mbox{if }
           B_2 (0) \cap D \neq \emptyset \\
&  \average _ { B_2 (0) \cap \Omega} u \, dy , \qquad && \mbox{if } B_2(0) \cap D  = \emptyset
. 
       \end{aligned}\right.
         $$
         \note{Give a proof or reference for the Poincare inequality below?

{\color{blue}           Proof in Croyle.}

         Changed earlier $\tilde D$ to $D'$.}         
     With this definition, we have the following variant of the Poincar\'e
     inequality (see the dissertation of Croyle \cite[pp.~38--39]{MR3564069}). There exists a constant $C$ so that 
\begin{equation}                    \label{PI}
         \| u - \bar u \| _ { L^2(\locdom 0 2 \cap \Omega) } \leq C 
           \| \nabla u \| _ { L^2 (\locdom 0 4 \cap \Omega)} ,\qquad u
           \in W^ {1,2}_D( B_4(0)\cap \Omega)
\end{equation}
and that $ v \in W^ {1,2} _{\tilde D} ( \reals ^ n _ +)$. Furthermore,
$v$ satisfies the mixed boundary value problem on $ \reals ^ n _ +$,
$$
\left \{ \begin{aligned}& \div A \nabla v - v = G , \qquad && \mbox{in }
  \reals ^ n _+\\
&  v=0, \qquad &&  \mbox{on } \tilde D \\
&  A \nabla v \cdot \nu = g _N , \qquad && \mbox{on } \tilde N
\end{aligned} \right.
$$
for some $G \in L^2 ( \reals ^ n _ +)$ and $ g_ N \in W^ { 1/2, 2} (
\reals ^ { n-1} )$. A tedious calculation gives that
$$
\| G \| _ { L^ 2 ( \reals ^ n _ +)} \leq C \| \nabla u \| _ { L^2 (
  \Omega _4 (0))}.
$$
This estimate makes use of the Poincar\'e inequality \eqref{PI}.
The Neumann data $ g_N$ is given by
$$
g _ N (y) = ( ( u -\bar u ) \frac { \partial \eta}{\partial \nu } \sqrt
{ 1 + |\nabla \phi|^2 } ) \circ \Phi. 
$$
Thus $g_N$ satisfies
$$
\| g_N \| _ { W^ { 1/2, 2} ( \reals ^ { n-1}) } \leq C \| \nabla u \|_
{ L^2 ( \Omega _ 4 (0)) } 
$$
where have  used  trace theorem   and the Poincar\'e inequality
\eqref{PI}.

Since we assume that $ 0 < r < r _0$, we may assume a uniform bound on
the $ C^ { 1,1 }$-norm of $\phi$. Hence we may apply Theorem
\ref{Savare} to estimate $ \nabla v $ in $B^ { 1/2} _ {2, \infty } (
\reals ^ n _ + )$ and then a Sobolev embedding theorem for Besov
spaces to conclude
$$
\left ( \average _{ \reals ^ n _+} |\nabla v |^ p \, dy \right) ^ {
  1/p}
\leq C \left ( \average _ { \Omega _4 (0) } |\nabla u |^2 \, dy \right
) ^ { 1/2}, \qquad 1 \leq p < 2n/( n-1).
$$
Finally, a change of variable leads to the estimate
$$
\left (\average _ { \locdom x r } |\nabla u |^ p \, dy \right ) ^ {
  1/p}
\leq \left (\average _ { \locdom x {2r }} |\nabla u |^ 2 \, dy \right
) ^ { 1/2}.
$$
From here, the techniques found in Giaquinta
\cite[pp.~80--82]{MR1239172}, for example, allow us to establish the
inequality with an $L^1$-average on the left-hand side. 
\end{proof}

\note{Probably need to dig up a reference for the embedding theorem
  for Besov Spaces.}
  
\section{Proof of  Theorem \ref{Main}}
We first observe that  it is known that we may solve the Dirichlet
problem with data in $\sobolev q 1(\partial \Omega)$ (commonly known
as the regularity 
problem)  and obtain non-tangential maximal
function estimates for the gradient for a larger range of indices than
we are considering here.  This will allow us to reduce to the case
when $f_D=0$. 
 See Dahlberg and Kenig
\cite{DK:1987} where results are given for $n\geq 3$ and $ 1 < q <
2+\epsilon$.  However, the result for $C^{1,1}$-domains is much easier
and is covered by the results for $C^1$-domains of Fabes, Jodeit, and
Rivi\'ere \cite{FJR:1978} as well as classical results such as
Kellogg.

Thus, we restrict our attention to the case $f_D=0$.
The proof of our main theorem in this case  relies on a real-variable
technique of Caffarelli and Peral \cite{MR1486629} which
Shen \cite[Theorem 3.2]{ZS:2007} adapted to the study of boundary value
problems.
We quote the result of Shen that is a key part of our argument. 

\begin{theorem}[{\cite[Theorem 3.2]{ZS:2007}}] \label{Shen} Let $Q_0$ be a cube in $\reals ^n$ and
  $F \in L^ { q_0   } ( 2 Q _0)$. Suppose that $ q _0 < q < q_1$ and $
  f \in L^ q ( 2   Q_0)$.  For each subcube $Q$  with  $ Q \subset Q_0$
  and $ |Q| <   \beta |Q_0|$, there exist functions $F_Q$ and $R_Q$ on
  $Q$ so that 
  \begin{gather}\label{ZS0}
    |F| \leq C(|F_Q| + |R_Q|) \\
\label{ZS2}    
    \left( \average _Q |R_Q|^ {q_1} \, d\sigma\right) ^ { 1/q_1}
    \leq C \left [ \left ( \average _ {2Q} |F|^ { q_0} \, dx \right) ^
      { 1/q_0}  + \sup _{ Q' \supset Q } \left ( \average _ { Q'} |f|^
      { q_ 0 } \, dx\right) ^ { 1/q_0 } \right]\\
\label{ZS1}    \left( \average _Q |F_Q|^ { q_0} \, dy\right) ^ { 1/q_0}
\leq C \sup _ { Q' \supset Q } \left ( \average _{Q'} |f|^ { q_0 } \,
dy \right )^ { 1/q_0}.
  \end{gather}
  With these assumptions, we have
  $$
  \left( \average _ {Q_0 } |F|^ q \, dy \right ) ^ { 1/q} 
  \leq C\left [ \left ( \average _ { 2 Q_0} |F|^ { q_0 } dy \right) ^
    { 1/q_0 }
    + \left ( \average _ {2Q_0 } |f|^ q \, dy \right ) ^ { 1/q} \right]. 
  $$
  Here, $ \beta <1$ and the constants $C$ in (\ref{ZS0}-\ref{ZS1})
  are independent of $f$ and $Q$. 
\end{theorem}

To apply this result, we will need to work on a set in $ \partial
\Omega$ which can be mapped to a cube in $ \reals ^ { n-1}$. We will
call these sets surface cubes and give a precise definition.  
 We recall our covering of $ \partial \Omega$ by  balls
as in (\ref{Coords}). If we fix a ball $B= B_{r_0} (x)$ so that $ \partial \Omega$ is
given by the graph of $ \phi$ in $ B$, we define a surface cube to be
the image of a cube in $ \reals ^ {n-1}$ under the map $ x'
\rightarrow ( x', \phi (x'))$. We also may define dilations of
boundary cubes $rQ$ (at least for $r$ small) by dilating the cube in $
\reals ^ { n-1}$.

Our next step  towards applying Theorem 
\ref{Shen}  is the following reverse H\"older
inequality at the boundary.
\begin{lemma}\label{Partway}
Let  $ \Omega$, $N$ and $D$ be a standard $C^{1,1}$-domain for the mixed
problem. 
Let 
$u$ be a weak solution of the mixed problem (\ref{WMP}) in $ \Omega$.
Assume that $ \nabla u ^ *  \in L^ { q_0} (
\partial \Omega)$ for some $ q_0 >1$.

If $u = 0 $ on $D \cap B _{2r} (x)$ and $ \partial u/ \partial \nu =
0$ on $N \cap B_ { 2r } (x)$, then for $q$ with $1< q < n/( n-1)$, we have
$$
\left( \average _ { \sball x r } |\nabla u | ^ q \, d\sigma \right) ^
{   1/q } \leq C \average _ {\locdom x {2r} } |\nabla u |\, dy.
$$
\end{lemma}

\begin{proof}
%%p 53  
Given $q$ with  $1< q < n/(n-1)$, we may choose $ s \in (0,1)$, but
close to 1,  so that $ p=(
1+s)q/s$ satisfies  $ 2< p < 2n/( n-1)$.
\note{ {\color{blue}With $s\in (0,1)$, we have $ (1+s)/s \in (2,\infty)$. We want $
(1+s)/s$ to be slightly larger than 2 so that $p$ is still less than $
2n/(n-1)$. In the argument below, we need $ s<1$ so that $ \delta ^ {
  -s}$ is integrable on the boundary.  }}
We fix $x \in \partial \Omega$ and $ r>0$ and let $ \Delta _ r =
\sball x r$. We begin by showing
\begin{equation}
  \label{Claim1}
  \left ( \average _ { \Delta _r } |\nabla u |^ q \, d\sigma \right )
  ^ { 1/q}
  \leq C \left ( \average _ { \Omega _ {2r}} |\nabla u |^ p \, dy
    \right) ^ { 1/p}.
\end{equation}
The first step to proving \eqref{Claim1} is to use  H\"older's
inequality to obtain that
$$
\left( \average _ { \Delta _r} |\nabla u |^ q \, d\sigma \right ) ^ {
  1/q}
\leq C
\left (\average  _{ \Delta _r } |\nabla u |^2 \, \delta ^ \alpha \,
d\sigma \right ) ^ { 1/2} \left ( \average _ { \Delta _r } \delta ^ {
  -\alpha q / ( 2-q) } \, d\sigma \right ) ^ { 1/q - 1/2}
$$
where $ \alpha $ is chosen so that $ \alpha q /( 2-q) =s$. Next, we
use  estimate of Ott and Brown \cite[Lemma 4.9]{MR3042705}, that $ \partial u
/\partial \nu =0$ on $ \Delta _r \cap N$ and then H\"older's
inequality to obtain 
\begin{equation*}
  \begin{split}
\left ( \average _ { \Delta _r } |\nabla u |^2 \delta ^ \alpha \,
d\sigma \right ) ^ { 1/2}
&\leq C r^ { 1/2} \left ( \average _ { \Omega _ {
    2r}} |\nabla u | ^ 2 \delta ^ { \alpha -1} \, dy \right ) ^ { 1/2}  \\
&\leq C r ^ { 1/2} \left ( \average _ { \Omega _ { 2r}} |\nabla u |^ p
  \, dy \right ) ^ { 1/p } 
  \cdot \left ( \average _ { \Omega _ {2r} } \delta ^ { ( \alpha -1 ) p
    / ( p-2) }\, dy \right ) ^ { 1/2 - 1/p }. 
  \end{split}
\end{equation*}
Using our definitions of $p$, $q$, and $\alpha$,  a calculation gives that  $ (
1-\alpha)p /(p-2)= 1+s$. Thus, we arrive at the estimate
\begin{equation}
  \label{Almost}
  \begin{split}
  \left( \average _ { \Delta _r } |\nabla u |^ q \, d\sigma \right) ^
       { 1/q} 
     &  \leq C r^ {1/2}
       \left ( \average _{ \Delta _r } \delta ^ { -s} \, d \sigma
       \right ) ^ { 1/q-1/2} \\
       &\qquad \times \left ( \average_{ \Omega _ { 2r}} \delta
         ^ {-s-1} \, dy \right ) ^ { 1/2 -1/p}
           \left ( \average _ { \Omega _ { 2r} } |\nabla u |^ p \, dy
           \right ) ^ { 1/p }.
  \end{split}
\end{equation}
From Lemmata 2.4 and 2.5 in Taylor {\em et.~al.~}\cite{MR3034453} we
have
\begin{align*}
\average _ { \sball x r } \delta ^ { -a} \,d \sigma \approx \max ( r,
\delta (x) ) ^ { -a} , \qquad a < 1.
\\
\average _ { \locdom  x r } \delta ^ { -b} \,d \sigma \approx \max ( r,
\delta (x) ) ^ { -b} , \qquad b < 2.
\end{align*}
Recalling that $ s/q - ( 1+s) /p =0$, we have
$$
r ^ { 1/2} \left( \average _{\Delta _r } \delta ^ { -s} \, d\sigma
\right ) ^ { 1/q- 1/2} \left ( \average _ { \Omega _ { 2r}} \delta ^ {
    -s-1} \, dy \right) ^ { 1/2 - 1/p} \approx r ^ { 1/2} \max
  (r,\delta(x))^ { -1/2} \leq C.
  $$
Using this, \eqref{Almost} and Theorem \ref{RHI} we obtain the
conclusion of the Lemma, except with $ \locdom x {4r}$ on the
right. A simple covering argument alllows us to obtain the result as
stated. 
\end{proof}
%%p. 58

Before continuing, we introduce several truncated
maximal functions. One appears in the  the next Lemma and the
remaining ones will be needed the proof of our main theorem.
The use of these auxiliary functions is needed to repair an error in
the work of Ott and Brown. The estimate (7.4) of \cite{MR3042705} is
not correct.
A correction is being prepared which uses a version of
the argument presented below. Thus, the results of Brown and Ott are
correct.

 We fix a small constant $c$ and a parameter  $r>0$. In applications,
 the value of $r$ will be clear from the context. The truncated
 non-tangential maximal functions are defined by 
\begin{equation}\label{TruncDef}
u^ \vartriangle (x) = \sup _ { y \in \Gamma (x) , |x-y | > cr }
|u(y) | , \qquad
u^ \triangledown (x) = \sup _ { y \in \Gamma (x) , |x-y | < cr }
|u(y) | .
\end{equation}
We will also need to introduce the Hardy-Littlewood maximal function
on $ \partial \Omega$ which we define as
$$
M(f)(x) = \sup _{ s>0} \average _{\sball x s } |f|\, d\sigma.
$$
In analogy with the truncated non-tangential maximal functions defined
in \eqref{TruncDef}, we will also define runcated versions of the
Hardy-Littlewood maximal function using the parameter $r$.  The truncated maximal functions
are defined by: 
$$
M_0(f)(x) = \sup _ { 0<s<r } \average _ { \sball x s } |f|\, d \sigma
\qquad M_ { \infty } (f) (x) = \sup _ { r \leq s } \average _ { \sball
  x s } |f|\, d \sigma.
$$

The next Lemma gives the value of $c$ that we will use in
(\ref{TruncDef}). 

\begin{lemma}\label{LocMaxFcn}
  Suppose that $u$ is a local solution  of the mixed problem
  $$\left \{ 
  \begin{aligned}
    &\Delta u =0 , \qquad && \mbox{in  } \locdom x { 2r}\\
    &u=0, \qquad && \mbox {on } D \cap \ball x { 2r}\\
    &\frac { \partial u }{ \partial \nu } = 0 ,  \qquad && \mbox{on }
    N \cap \ball x { 2r}
  \end{aligned}
  \right. 
  $$
  Then given $q$ in $(1,\infty)$,  there exists $c >0$ so that  with $
  \nabla u^ \triangledown $ as in \eqref{TruncDef} we have
  $$
  \left ( \average _{ \sball x r } (\nabla u ^\triangledown)^q    \,
  d\sigma \right ) ^ { 1/q }  \leq C \average _ { \sball x { 2r } }
  \nabla u ^* \, dy .  
    $$
\end{lemma}

\begin{proof} We establish a representation formula for $ \nabla u$
  and apply the result of Coifman, McIntosh and Meyer \cite{CMM:1982}
  as in the work of Ott and Brown 
  \cite[Section 6]{MR3042705} to conclude that
\begin{equation} \label{LMF.1}
  \left ( \average _{ \sball x r } ( \nabla u ^ * ) ^ { q} \, d\sigma
  \right ) ^ { 1/q}  \leq C\left [ \average _ {\locdom x {2r} } |\nabla u | \,
    dy
    + \left ( \average_ { \sball x {2r} } |\nabla u | ^ q \, d\sigma\right )^ {
      1/q}\right].
\end{equation}
We may use Lemma \ref{Partway} to bound the  second term  on the right
of \eqref{LMF.1} and a standard   argument gives that there is a
constant $C$ so that 
\begin{equation}\label{LMF.2}
\int _ { \locdom x { 2r} } |\nabla u | \, dy
\leq C r \int _ { \sball x { Cr} } \nabla u ^ * \, d \sigma.
\end{equation}
Combining \eqref{LMF.1} and \eqref{LMF.2}, we obtain the desired result
with $ \sball x { Cr}$  rather than $ \sball x {2r}$ on the right. We
may obtain the stated result by a simple covering argument. This may
require us to decrease the value of the constant $c$ used in the
definition of $ \nabla u ^ \triangledown$. 
\end{proof}

We now give two Lemma related to the truncated maximal functions.

  \begin{lemma} \label{Minf}
    Suppose that $x,y$ are in $ \partial \Omega$ and $ |x-y| < Ar$,
    then we have
    $$ M_\infty (f) (x) \leq C_A M_ \infty (f) (y).
    $$
  \end{lemma}
  \begin{proof} By the triangle inequality, we have $ \sball x s
    \subset \sball y { s+Ar}$. Thus it follows that
    $$
    \average _ { \sball x s } |f|\, d\sigma
    \leq  \frac { \sigma ( \sball y { s+Ar})}{\sigma (\sball x s )}
    \average_{ \sball y { s+Ar}} |f|\, d\sigma. 
    $$
    If we require that $s \geq r$, then
    we have a constant so that
    $ \sigma ( \sball y { s+Ar})/\sigma (\sball x s ) 
\leq C_A$ which gives the Lemma. 
  \end{proof}

  \begin{lemma} \label{Far} We have 
    $$
    u ^ \vartriangle (x) \leq C M _ {\infty }(  u^* ) (x).
    $$
    The constant depends on the value of $c$
    entering into the definition of $ u ^ \vartriangle$.  
  \end{lemma}

  \begin{proof} Fix $x\in \partial \Omega$ and suppose that $y \in
    \Gamma (x)$. Fix $ \hat y$ so that $ |y-\hat y| = d(y) = \dist(y,
    \partial \Omega)$ and observe that if $ |z-\hat y | < \alpha d(y)$, we
    have $y \in \Gamma (z)$. This implies $ |u(y) | \leq u^ * (z) $ 
    for $ z \in \sball { \hat y } { \alpha d(y)}$.
    By the triangle inequality $
    |x-\hat y| \leq |x-y | + |y - \hat y| \leq (2+ \alpha )
    d(y)$. Hence we have that $\sball { \hat y} {\alpha d (y)} \subset
    \sball x { ( 2 + 2 \alpha ) d(y)} $. It follows that
    $$
    |u(y) | \leq 
     \frac { \sigma ( \sball x { (2 + 2 \alpha )d(y)} )}
    { \sigma ( \sball {\hat y} { \alpha d(y) } ) } \average _ { \sball x { (
        2+ 2 \alpha ) d(y)}} u^ * \, d\sigma  . 
    $$
    %%Eliminated C from this display after sending draft to
    %%croyle. 2/15. 
If, in addition, we assume that $ |x-y| >cr$, then we will have $ d(y)
> cr/(1+ \alpha)$ and  we obtain
$$
u^ \vartriangle (x) \leq C M _ \infty ( u ^ *) ( x).
$$
  \end{proof}

\begin{proof}[{Proof of Theorem \ref{Main}}]
We fix $ f_N \in L^q (N)$ with $ 1 < q < n/(n-1)$ and let $ u$ the
solution of
$$
\left\{
\begin{aligned}
&  \Delta u = 0, \qquad && \mbox{in } \Omega \\
&  u = 0 , \qquad && \mbox{on } D \\
&  \frac { \partial u } { \partial \nu } = f_N , \qquad && \mbox{on } N\\
&  \nabla u ^ * \in L^ { q_0} (\partial \Omega)
\end{aligned}
\right.
$$
According to Theorem 1.2 of  Ott and Brown \cite{MR3042705}
 there is an index $q_0$ with $1<q_0<q$ for which we can solve this
boundary value problem and find $u$.

We fix a surface cube $Q_0 \subset \partial \Omega$ and suppose $
\partial \Omega$ is given by a graph in $ 2 Q_0$. 
We will show that
\begin{equation}\label{John}
  \int _ { Q_0} M(\nabla u ^ *)^ q \, d\sigma \leq C \int _ {
    \partial \Omega} |f_N|^q \, d\sigma.
\end{equation}
If we cover $ \partial \Omega$ by a finite collection of surface
cubes and sum the resulting estimates we will obtain an estimate 
 for the maximal function $
\nabla u ^ *$.

Thus, we turn to the proof of \eqref{John}. To verify the hypotheses
of 
  Theorem \ref{Shen},  we  fix  a cube $Q \subset Q_0$ and define $v$ and $w$
in $\Omega$ as the solutions of the boundary value problems
\begin{align*}
  \left\{
  \begin{aligned}
  &  \Delta v = 0 , \qquad && \mbox{in } \Omega \\
&    v=0 , \qquad && \mbox{on } D \\
&    \frac { \partial v } {\partial \nu} =g , \qquad &&\mbox{on }N\\
 &   \nabla v ^ * \in L^ { q_0}(\partial \Omega ) 
  \end{aligned}
  \right.
 \qquad  &
  \left\{
    \begin{aligned}
 &     \Delta w = 0 , \qquad &&\mbox{in } \Omega \\
&    w=0 , \qquad && \mbox{on } D \\
 &   \frac { \partial w } {\partial \nu} =h , \qquad &&\mbox{on }N\\
  &  \nabla w ^ * \in L^ { q_0}(\partial \Omega )       
  \end{aligned}
    \right.
\end{align*}
where $ g = \chi _ { 2Q} f_N$ and $ h = f_N -g$. In preparation for
using Theorem \ref{Shen}, we put $F= M(\nabla u^*)$, $ F_Q = M(\nabla v ^ *)$ and
$R_Q = M(\nabla w ^*)$.  By uniqueness for the $L^
{q_0}$-mixed problem \cite[Theorem 5.1]{MR3042705},  we have that $u = v
+w$ and it follows that \eqref{ZS0} holds on $Q$.

To prove \eqref{ZS1} %%msref p. 49
we use Theorem 7.7 of Ott and Brown \cite{MR3042705} and the
Hardy-Littlewood maximal Theorem  to conclude that
\begin{equation}
\int _ {\partial \Omega}  M(\nabla v ^*) ^ {  q_0} \, d\sigma
\leq C \int _ { 2Q} |f_N|^ { q_0}\, d\sigma \label{E65}
\end{equation}
%%msref p. 65
The estimate \eqref{ZS1} follows easily from \eqref{E65}. 

\note{Need to check out dilating of cubes, etc.

Give context for r etc.}
The  proof of  \eqref{ZS2} will require a bit more work. We begin by
choosing $ r > c \diam (Q)$ so that if $ x \in Q$, then $ \sball x
{4r} \subset 2Q$. This will be the value of $r$ we use in defining our
truncated maximal functions. We claim that for $ q _ 1 < n / ( n-1)$, we have
\begin{equation}
  \label {NS}
  \left ( \average _ { \sball x r } M ( \nabla w ^ *) ^ { q_1} \,
  d\sigma \right ) ^ { 1/q_1}
  \leq C \left[ \average _ { \sball x { 4r}} M ( \nabla u ^ * ) \, d\sigma
      + \left( \average _ {2 Q}  |f_N |^ { q_0 } \, d\sigma \right) ^
      { 1/q_0}\right ]
\end{equation}
We may obtain \eqref{ZS2} from \eqref{NS} by covering $Q$ with a
finite number of surfaces balls. 

To prove \eqref{NS} we begin by observing that 
$$
M(\nabla w ^ *) \leq M_ \infty(\nabla w ^ * ) + M_0 ( \nabla w  ^
\triangledown ) + M_0 ( \nabla w ^ { \vartriangle}) . 
$$
According to Lemma \ref{Minf}  we have
\begin{equation} \label{MT.A}
   M _ \infty (\nabla  w^ * )(y) \leq C \average _ {
    \sball x r } M ( \nabla w ^ * ) \, d\sigma , \qquad y \in  \sball x r.
\end{equation}
The estimate
\begin{equation}\label{MT.B}
  \left ( \average _{\sball x r } M_0 ( \nabla w ^ \triangledown)^ {
    q_1} \, d\sigma\right ) ^ { 1/q_1}
  \leq C \average_{\sball x r } M( \nabla w ^ *)\, d\sigma
\end{equation}
follows from Lemma \ref{LocMaxFcn}.

To estimate $M( \nabla w ^ \vartriangle )  $ we use Lemma \ref{Minf}
and Lemma \ref{Far} to conclude that
$$
\nabla w ^ \vartriangle  (y) \leq C \average _ {\sball x {2r}} M _
\infty ( \nabla w ^ * ) \, d\sigma , \qquad y \in \sball x { 2r} .
$$
From this, we conclude that
\begin{equation} \label{MT.C}
  M _ 0 ( \nabla w ^ \vartriangle ) (y)
  \leq \average _ { \sball x { 2r} } M ( \nabla w ^ * )\, d\sigma,
  \qquad y \in \sball x { 2r} .
\end{equation}
Combining (\ref{MT.A}-\ref{MT.C}), we conclude 
\begin{equation}\label{MT.D}
 \left ( \average _{ \sball x r } M ( \nabla w ^ * ) ^ { q _1} \, d\sigma\right )
  ^ { 1/q_1}
  \leq C\left [ \average _{ \sball x { 4r}} M ( \nabla w ^ * ) \, d\sigma
    + \left( \average _ {2Q} |f_N|^ { q_0} \, d\sigma\right ) ^ {
      1/q_0}\right ]. 
\end{equation}
Since $w= u-v$, it follows that $ M( \nabla w ^ *) \leq M ( \nabla u
^*) + M ( \nabla v ^ *) $ we may use \eqref{E65} to obtain
\begin{equation}\label{MT.E}
  \average _ { \sball x { 4r}} M ( \nabla w ^ *) \, d\sigma
  \leq C \left[ \average  _ { \sball x {4r}}  M ( \nabla u^ *) \, d\sigma
    + \left( \average _ {2 Q} |f_N|^ { q_0 } \, d\sigma\right ) ^ { 1/q_0}\right ] .
\end{equation}
The estimate \eqref{NS} follows from \eqref{MT.D} and \eqref{MT.E}.
\end{proof}

We close with several questions for further investigation.

\begin{enumerate}
  \item As observed above, the range of exponents is sharp when the
    dimension $n=2$. In dimensions $n \geq 3$, there is a gap between
    the result of Theorem \ref{Main} and our example (extended to
    higher dimensions by adding extra variables). 

    \item In principle, the argument here should extend to
      systems. However, we have not written out the details.
\end{enumerate}

%%%%%%%%%%%%%%%%%%%%%%%%%%%%%%%%%%%%%
%Bibliography
%%%%%%%%%%%%%%%%%%%%%%%%%%%%%%%%%%%%%
\newpage
\bibliographystyle{plain}

%%Refs from new moved to main.

%%\bibliography{main}
%%\bibliography{new,main}
\def\cprime{$'$} \def\cprime{$'$} \def\cprime{$'$} \def\cprime{$'$}
  \def\cprime{$'$} \def\cprime{$'$} \def\cprime{$'$} \def\cprime{$'$}
  \def\cprime{$'$} \def\cprime{$'$}

\end{document}